\documentclass[11pt,a4paper]{article}
\usepackage{amsmath,amsfonts,amsthm,amssymb, mathtools}
\usepackage{mathrsfs, graphicx,color,latexsym, tikz, calc,
}
\usepackage{tikz-cd}
\usepackage{graphicx}
\usepackage{epstopdf}
\usepackage{enumerate}
\usepackage{caption}
\usepackage{stmaryrd}
\usepackage{hyperref}

\usetikzlibrary{shadows}
\usetikzlibrary{patterns,arrows,decorations.pathreplacing}
\usepackage{ulem}
\newcommand{\qbinom}[2]{\genfrac{[}{]}{0pt}{}{#1}{#2}}

\newtheorem{theorem}{Theorem}[section]
\newtheorem{lemma}[theorem]{Lemma}

\newtheorem{problem}[theorem]{Problem}

\newtheorem{claim}{Claim}
\newtheorem*{claim*}{Claim}

\allowdisplaybreaks[4]

\title{\bf \Large }
\date{ }
\title{\bf \Large 
Non-uniform pairwise cross $t$-intersecting families\footnote{E-mail addresses: 
\url{ytli0921@hnu.edu.cn} (Y. Li),
\url{mathtzwu@163.com} (T. Wu) 
}}
\author{
{\small  Yongjiang Wu$^1$, \ \ Yongtao Li$^2$,  \ \ Tingzeng Wu$^{3,4}$, \ \ Lihua Feng$^1$}\\[2mm]
\small $^1$School of Mathematics and Statistics, HNP-LAMA, Central South University\\
 \small Changsha, Hunan, 410083, China\\ 
 \small $^2$Yau Mathematical Sciences Center, Tsinghua University, Beijing, 100084, China\\
 \small $^3$School of Mathematics and Statistics, Qinghai Minzu University\\
 \small Xining, Qinghai, 810007, China\\
\small $^4$Qinghai Institute of Applied Mathematics, Xining, Qinghai, 810007, China\\
}

\begin{document}
\maketitle
\begin{abstract}
Let $ n\geq  t\geq  1$ and $ \mathcal{A}_1, \mathcal{A}_2, \ldots, \mathcal{A}_m \subseteq 2^{[n]}$ be non-empty families. We say that they are pairwise cross $t$-intersecting if $|A_i\cap A_j|\geq  t$ holds for any $A_i\in \mathcal{A}_i$ and $A_j\in \mathcal{A}_j$ with $i\neq j$. 
In the case where $m=2$ and $\mathcal{A}_1=\mathcal{A}_2$, 
determining the maximum size $M(n,t)$ of a non-uniform $t$-intersecting family of sets over $[n]$ was solved by Katona (1964), and  enhanced by Frankl (2017), and recently by Li and Wu (2024).  
In this paper, we establish the following upper bound: if  $ \mathcal{A}_1, \mathcal{A}_2, \ldots, \mathcal{A}_m \subseteq 2^{[n]}$ are non-empty pairwise cross $t$-intersecting families, then
$$
\sum_{i=1}^m |\mathcal{A}_i| \leq \max \left\{ \sum_{k=t} ^{n}\binom{n}{k}  + m - 1, \, m M(n, t) \right\}. 
$$
 Furthermore, we provide a complete characterization of  the extremal families that achieve the bound.
Our result not only generalizes an old result of Katona (1964) for a single family, but also extends a theorem of Frankl and Wong (2021) for two families. Moreover, our result could be viewed as a non-uniform version of a recent theorem of Li and Zhang (2025). The key in our proof is to utilize the generating set method and the pushing-pulling method together.  \\

\end{abstract}

\section{Introduction}

Extremal set theory studies the maximum or minimum possible size of a family of sets under given constraints. The central problems in this field include determining the extremal size of set families that satisfy certain intersection, union, or exclusion conditions, and characterizing the optimal structures. Classical results include the Erd\H{o}s--Ko--Rado theorem, Sunflower lemma, and others, employing combinatorial, probabilistic, and algebraic methods. Extremal set theory has wide applications in theoretical computer science, graph theory, coding theory, and discrete geometry, with the aim of uncovering the limits and optimal configurations of set systems.  

We fix the following notation. 
For integers $m\leq n$, let  $[m,n]=\{m, m+1, \ldots, n\}$ and simply let $[n]=[1,n]$. Let $2^{[n]}$ denote the family of all subsets of $[n]$, and let  $\binom{[n]}{k}$ denote the family of all $k$-element subsets of $[n]$.  Let $\binom{[n]}{\leq k}=\{A\subseteq [n]: |A|\leq k\}$. A family $\mathcal{A}\subseteq 2^{[n]}$  is called $k$-\textit{uniform} if all sets of $\mathcal{A}$ contain $k$ elements, otherwise it is called \textit{non-uniform}. 
All families discussed in this study are assumed to be non-empty. 
 A family $\mathcal{A} \subseteq 2^{[n]}$ is called $t$-\textit{intersecting} if $|A\cap A^{\prime}|\geq  t$
 for all $A, A^{\prime}\in \mathcal{A}$. 
Two families $ \mathcal{A}, \mathcal{B} \subseteq 2^{[n]} $ are called  \textit{cross $t$-intersecting} if $|A\cap B|\geq  t$
 for all $A\in \mathcal{A}$ and $B\in \mathcal{B}$.
Moreover, the families $ \mathcal{A}_1, \mathcal{A}_2, \ldots, \mathcal{A}_m \subseteq 2^{[n]}$  are called  \textit{pairwise cross $t$-intersecting} if $|A_i\cap A_j|\geq  t$ for any $A_i\in \mathcal{A}_i$ and $A_j\in \mathcal{A}_j$ with $1\leq i<j\leq m$. 
For the case $t=1$, we abbreviate $1$-intersecting, cross $1$-intersecting and pairwise cross $t$-intersecting as intersecting, cross intersecting and pairwise cross intersecting, respectively.

\subsection{Uniform $t$-intersecting families}

Let $M(n, k, t)$ denote the maximum of $|\mathcal{A}|$ over all $\mathcal{A}\subseteq \binom{[n]}{k}$ having the $t$-intersecting property.
The classical Erd\H{o}s--Ko--Rado theorem \cite{E61} states that $M(n,k,1)={n-1 \choose k-1}$ for every $k\geq 2$ and $n\geq 2k$. 
For any $t\geq 2$ and $k\geq t$, there exists an integer $n_0(k,t)$ such that if $n>n_0(k, t)$, then 
\[ M(n, k, t)=\binom{n-t}{k-t}, \] 
and the unique extremal family is $\mathcal{A}=\{A\in {[n] \choose k}: T\subseteq A\}$ for some fixed $T\in {[n ] \choose t}$.  The original paper \cite{E61} established that $n_0(k,t) \leq (k-t) {k \choose t} {}^3 +t$. 
The smallest such $n_0(k,t)$ was conjectured to be $(t+1)(k-t+1)$, which was confirmed by Frankl \cite{F78} for $t\geq  15$ using the random walk method, and completely solved by Wilson \cite{W84} for all $t\geq 1$ by the algebraic method. 

For small integer $n$, Ahlswede and Khachatrian introduced the generating set method \cite{AK1997} and the pushing--pulling method \cite{AK1999}. 
They obtained the exact value of $M(n, k, t)$ for every $n < (t+1)(k-t+1)$. This solves a long-standing conjecture of Frankl \cite{F78}. 

\begin{theorem}[Ahlswede and Khachatrian \cite{AK1997,AK1999}] \label{AK1997}
Let $n, k, t, r$ be integers with $n \geq  k \geq  t\geq  1$ and $n>2k-t$. We denote $\mathcal{F}_r(n,k,t)  =\left\{A \in \binom{[n]}{k}:|A\cap[t+2r]| \geq  t+r\right\}$. Then
\begin{equation*} \label{eq-AK}
M(n,k, t)= \max_{0\leq r\leq {(n-t)}/{2} }\{| \mathcal{F}_r(n,k,t)|\}. 
\end{equation*} 
\end{theorem}

To be more specific, it is easy to see that 
$|\mathcal{F}_r(n,k,t)| = \sum_{i=t+r}^{t+2r} {t+2r \choose i} {n-t-2r \choose k-i}$. 
For more related results for $t$-intersecting families, 
we refer to \cite{AK1996,BM2008, CLW2021, CLLW2022,ZW2025,FW2025} and references therein. 

In 1967, Hilton and Milner \cite{H67} pioneered the investigation of cross intersecting families in analogy to the celebrated Erd\H{o}s--Ko--Rado theorem, establishing that for non-empty cross intersecting families $\mathcal{A}, \mathcal{B} \subseteq \binom{[n]}{k}$ with $n \geq  2k$, the sharp bound $|\mathcal{A}| + |\mathcal{B}| \leq \binom{n}{k} - \binom{n-k}{k} + 1$ holds. This foundational result was generalized in \cite{F92,FT1998,Fra2024,FW2024-EUJC}. 
More generally, 
Wang and Zhang \cite{W13} determined the maximum sum of sizes of non-empty cross $t$-intersecting families. 

\begin{theorem}[Wang and Zhang \cite{W13}]
Let $n,k,t$ be positive integers with $k> t\geq  1$ and $n> 2k-t$. 
If $\mathcal{A}, \mathcal{B} \subseteq {[n] \choose k}$ 
 are non-empty cross $t$-intersecting families, then 
 \begin{equation*}
     \label{eq-WZ}
      |\mathcal{A}| + |\mathcal{B}| \leq {n \choose k} 
 - \sum_{i=0}^{t-1} {k \choose i} {n-k \choose k-i} +1. 
 \end{equation*}
 \end{theorem}
 This bound was also independently obtained by Frankl and Kupavskii \cite{F17}.

 In 2022, 
 Shi, Frankl and Qian \cite{S22} generalized  Hilton--Milner's result to the multi-family setting, proving that for non-empty pairwise cross intersecting families $\mathcal{A}_1, \mathcal{A}_2, \ldots, \mathcal{A}_m \subseteq \binom{[n]}{k}$ with $n \geq  2k$ and $m \geq  2$, the total size satisfies $\sum_{i=1}^m |\mathcal{A}_i| \leq \max \big\{ \binom{n}{k} - \binom{n-k}{k} + m - 1, \, m \binom{n-1}{k-1} \big\}$. 
 This result stimulated more involved studies on this topic; see  \cite{HY-first,H25,Huang2025,ZF2024} for non-empty pairwise cross intersecting families in which the families have different uniformity.

 Solving an open problem of Shi, Frankl and Qian \cite[Problem 4.4]{S22}, 
 Li and Zhang  \cite{L25} recently proved the following result for  families with pairwise cross $t$-intersecting property.

\begin{theorem}[Li and Zhang \cite{L25}]\label{L25}
Let $ n, k, t$ and $ m $ be  integers with $ n > 2k - t,  k > t \geq  1$ and $ m \geq  2 $. If $ \mathcal{A}_1, \mathcal{A}_2, \ldots, \mathcal{A}_m \subseteq \binom{[n]}{k} $ are non-empty pairwise cross $t$-intersecting families, then

$$
\sum_{i=1}^m |\mathcal{A}_i| \leq \max \left\{ \binom{n}{k} - \sum_{i=0} ^{t-1}\binom{k}{i} \binom{n-k}{k-i} + m - 1, \, m M(n, k, t) \right\}. 
$$
\end{theorem}

It remains an open problem \cite[Problem 4.5]{S22} to extend Theorem \ref{L25} to multi-families where the sets of families have different cardinalities, i.e., 
$\mathcal{A}_i \subseteq {[n] \choose k_i}$ for every $i\in [m]$, where $k_1\geq k_2\geq \cdots \geq k_m$.  
In this paper, we shall investigate Theorem \ref{L25} for  non-uniform families in another perspective. We shall prove an analogue for pairwise cross $t$-intersecting families $\mathcal{A}_1, \mathcal{A}_2,\ldots ,\mathcal{A}_m\subseteq 2^{[n]}$. This non-uniform extension in our paper may be of independent interest.

\subsection{Non-uniform $t$-intersecting families}

Now, we consider the maximum size of a non-uniform $t$-intersecting family. 
Let $M(n, t)$ denote the maximum of $|\mathcal{A}|$ over all $\mathcal{A}\subseteq 2^{[n]}$ having the $t$-intersecting property.
The following classical result of Katona \cite{K64} determines the exact value of 
 $M(n, t)$ for all $n, t$. 

\begin{theorem}[Katona \cite{K64}] \label{K64}
Let $n, t$ be integers with $n \geq  t\geq  1$.  Then
$$
M(n, t)= \begin{cases} 
\sum_{k\geq \frac{n+t}{2}} {n \choose k} & 
\text{if $n+t$ is even}; \\ 
\sum_{k\geq \frac{n+t+1}{2}} {n \choose k} + {n-1 \choose \frac{n+t-1}{2}} & \text {if $n+t$ is odd}.\end{cases}
$$ 
Moreover, for $t>1$, the first equality attains only for $ \mathcal{K}(n,t) =\left\{A \subseteq [n]:|A| \geq  \frac{n+t}{2}\right\}$; the second equality attains only for $\mathcal{K}'(n,t)=\left\{A \subseteq [n]:|A| \geq  \frac{n+t+1}{2}\right\}\cup {[n-1] \choose \frac{n+t-1}{2}}$.
\end{theorem}

The odd case of Katona's theorem can be deduced from the even case \cite{AK1999}. Moreover, the odd case of Katona's theorem implies the celebrated Erd\H{o}s--Ko--Rado theorem for uniform intersecting families. 
The original proof of Theorem \ref{K64} is based on using the so-called Katona intersection shadow theorem. 
Alternative proofs of Theorem \ref{K64} can be found in \cite{AK1999,AK2005,Wang1977} and \cite[page 98]{Bol1986}. We refer to \cite{Fra2017jctb,Fra2017cpc,HKP2020,LW2024,WFL2025} for  recent related extensions.

In 2021, 
 Frankl and Wong \cite{FW21} established an extension of Katona's theorem.

\begin{theorem}[Frankl  and Wong \cite{FW21}] \label{FW21}
Let $n, t$ be integers with $n \geq  t\geq  1$.  Let $\mathcal{A},\mathcal{B} \subseteq 2^{[n]}$ be  non-empty cross $t$-intersectng families. Then
$$
|\mathcal{A}|+|\mathcal{B}| \leq \sum_{i=t}^n\binom{n}{i}+1.
$$
Moreover, for $t>1$, the equality holds if and only if $\{\mathcal{A}, \mathcal{B}\}=\{[n],\{G \subseteq [n]: |G|\geq  t\}\}$.
\end{theorem}

This result was later expanded by Borg and Feghli \cite{B22}, who investigated cross intersecting families in a more general non-uniform setting, specifically for $\mathcal{A} \subseteq \binom{[n]}{\leq k}$ and $\mathcal{B} \subseteq \binom{[n]}{\leq \ell}$. Liu \cite{L23} further advanced this line of research by considering the cross $t$-intersecting case, providing a broader framework for such extremal problems. For further developments and related results, we refer to \cite{F17, FT18, G23, H25, JXXZ2024}, as well as the comprehensive survey \cite{FT16}.

Our result presents a unified generalization of both Theorem \ref{K64} and Theorem \ref{FW21}, and it could also be viewed as an extension of Theorem \ref{L25} to the non-uniform families.

\begin{theorem}[Main result] \label{ma1}
Let $ n\geq  t\geq  1$ and $ m\geq  2$ be  integers. If $ \mathcal{A}_1, \mathcal{A}_2, \ldots, \mathcal{A}_m \subseteq 2^{[n]}$ are non-empty pairwise cross $t$-intersecting families, then

$$
\sum_{i=1}^m |\mathcal{A}_i| \leq \max \left\{ \sum_{k=t} ^{n}\binom{n}{k}  + m - 1, \, m M(n, t) \right\}.
$$
For $t>1$, the equality holds if and only if, up to isomorphism, one of the following holds:
\begin{enumerate}
    \item[\rm (a)]  $\mathcal{A}_1=\left\{A \subseteq [n]:|A \cap[n]| \geq  t\right\}$ and $\mathcal{A}_2=\cdots= \mathcal{A}_m=\{[n]\}$ if $\sum_{k=t} ^{n}\binom{n}{k}  + m - 1\geq  m M(n, t)$.

\item[\rm (b)]  $\mathcal{A}_1=\cdots= \mathcal{A}_m= \mathcal{K}(n,t)$ for $2\mid (n+t)$  and $\mathcal{A}_1=\cdots= \mathcal{A}_m= \mathcal{K}'(n,t)$ for $2\nmid (n+t)$ if $\sum_{k=t} ^{n}\binom{n}{k}  + m - 1\leq m M(n, t)$.
\end{enumerate}
\end{theorem}

We point out here that Theorem \ref{ma1} implies 
Theorem \ref{K64} in a surprising way. Indeed, taking $m$ as sufficiently large and setting $\mathcal{A}_1= \cdots = \mathcal{A}_m$, we can see that $m M(n,t) > 
\sum_{k=t}^n {n \choose k} + m-1$. Thus, we get $|\mathcal{A}_1| \leq M(n,t)$ for every $t$-intersecting family $\mathcal{A}_1\subseteq 2^{[n]}$. 
Moreover, in the case $t=1$, there are many optimal families. For example, for the case $2\mid (n+1)$, the families $\mathcal{A}_1=\cdots= \mathcal{A}_m= \mathcal{K}(n,1)$  and $\mathcal{A}_1=\cdots= \mathcal{A}_m= \left\{A \subseteq [n]:1\in A\right\}$ all have $\sum_{i=1}^m |\mathcal{A}_i|=m2^{n-1}$.

  \medskip 
\noindent 
{\bf Organization.} 
 The remainder of this paper is organized as follows: In Section \ref{se2}, we introduce the key ingredient in our proof, which combines the generating set method \cite{AK1997} with the pushing-pulling method \cite{AK1999}. In Section \ref{se3}, we present the complete proof of Theorem \ref{ma1}.

\section{Preliminaries} \label{se2} 
Throughout this paper, we adopt the following notation.
Let $\mathcal{A} \subseteq 2^{[n]}$ be a family. 
We denote by $\mathcal{A}^{(k)}=\{A\in\mathcal{A}: |A|=k\}$ the $k$-uniform subfamily of $\mathcal{A}$. 
For any $x\in [n]$, we define 
\begin{align*}
 \mathcal{A} [x]=\{A\in \mathcal{A}: x\in A\} 
 \quad \text{and} \quad  \mathcal{A} (x)=\{A\backslash\{x\}: x\in A\in \mathcal{A}\}.
\end{align*}
With the above notation, we  have
$ 
\mathcal{A}[x]^{(k)}=\{A\in\mathcal{A}: x\in A, |A|=k\}
$ 
and 
$ 
\mathcal{A}[x]^{(k)}(x)=\{A\backslash\{x\}: x\in A\in \mathcal{A}, |A|=k\}. 
$ Moreover, we denote $\mathcal{A}(\bar{x})=\{A\in\mathcal{A}: x\notin A\}$. 

\subsection{The shifting operator}

We begin with the shifting technique \cite{E61}, which is one of the most powerful tools in extremal set theory.
Let  $\mathcal{A} \subseteq 2^{[n]}$ be a family and let $1 \leq i<j \leq n$ be integers. The \textit{shifting operator} $s_{i, j}$ on $\mathcal{A}$ is defined as follows:
$$s_{i, j}(\mathcal{A})=\left\{s_{i, j}(A): A \in \mathcal{A}\right\},$$
where
$$
s_{i, j}(A)= \begin{cases}
(A \backslash\{j\}) \cup\{i\} & \text { if } j \in A, i \notin A \text { and } (A \backslash\{j\}) \cup\{i\} \notin \mathcal{A}; \\ A & \text { otherwise. }\end{cases}
$$
A family $\mathcal{A} \subseteq 2^{[n]}$ is called \textit{shifted} or \textit{left-compressed} if $s_{i, j}(\mathcal{A})=  \mathcal{A}$ holds for all $1 \leq i<j \leq n$.  
It is well-known that by applying the shifting operator repeatedly, every family will terminate at a shifted family (although the terminated family is not unique). A critical observation in the proof of the Erd\H{o}s--Ko--Rado theorem is that the shifting operation preserves the property of a family being intersecting; see the surveys  \cite{Frankl1987, Wang2025}. In fact, the shifting operator has the following important property for corss $t$-intersecting families; see, e.g., \cite{G23}.

\begin{lemma} \label{G23}
 Let $\mathcal{A},\mathcal{B}  \subseteq 2^{[n]}$  be  cross $t$-intersecting families and let $1 \leq i<j \leq n$ be integers.
Then $s_{i, j}(\mathcal{A})$ and $s_{i, j}(\mathcal{B})$ are cross $t$-intersecting.
\end{lemma}

A family $\mathcal{A}\subseteq 2^{[n]}$ is called an \textit{antichain} if there are no distinct sets $A, A'\in\mathcal{A}$ satisfying $A\subseteq A'$.
A family $\mathcal{A}\subseteq 2^{[n]}$ is called \textit{monotone} if $A\in\mathcal{A}$ and $A\subseteq B$ imply $B\in\mathcal{A}$. 
The \textit{up-set} of a family $\mathcal{A}$ is defined by
$
\langle\mathcal{A}\rangle=\left\{F\subseteq [n]: \text{ there exists }A \in\mathcal{A}\text{ such that }A\subseteq F\right\}.
$

\subsection{Generating set method}

The generating set method was initially introduced by  Ahlswede and Khachatrian \cite{AK1996, AK1997} and subsequently advanced and applied in recent works \cite{Fil2017, L25,ZW2025}.
Let $\mathcal{A}\subseteq 2^{[n]}$ ($\emptyset\subsetneq\mathcal{A}\subsetneq 2^{[n]}$) be a monotone family. A \textit{generating set} of $\mathcal{A}$ is an inclusion-minimal   set  in $\mathcal{A}$. 
The \textit{generating family} of $\mathcal{A}$ consists of all generating sets of $\mathcal{A}$. Note that any generating family is an antichain.
 The \textit{extent} of $\mathcal{A}$ is defined as the maximal element appearing in a generating set of $\mathcal{A}$. The \textit{boundary generating family} of $\mathcal{A}$ consists of all generating sets of $\mathcal{A}$ containing its extent.

The following lemma is inspired by  \cite[Lemma 2.2]{F24}.

\begin{lemma}\label{le1}
Let $\mathcal{A}\subseteq 2^{[n]}$ be a monotone shifted family with its extent $\ell\geq  2$, generating family $\mathcal{G}$ and boundary generating family $\mathcal{G}[\ell]$.  For any families $\mathcal{B}, \mathcal{C}\subseteq \mathcal{G}[\ell]$, we denote $\mathcal{D}=(\mathcal{G}\setminus(\mathcal{B}\cup\mathcal{C}))\cup\{C\setminus\{\ell\}: C\in\mathcal{C}\}$. Then
$$
{ |\langle\mathcal{D}\rangle|=|\mathcal{A}|-|\mathcal{B}\backslash\mathcal{C}|2^{n-\ell}+|\mathcal{C}|2^{n-\ell}.}
$$
\end{lemma}
\begin{proof}
For any $B\in \mathcal{B}$ and $C\in \mathcal{C}$, we have
 $B, C \subseteq [\ell]$. Hence,   $n-|B| \ge n-\ell$ and $n-|C|+1 \ge n-\ell+1 > n-\ell$.
In view of $|\langle\mathcal{D}\rangle|-|\mathcal{A}|=|\langle\mathcal{D}\rangle\setminus\mathcal{A}|-|\mathcal{A}\setminus\langle\mathcal{D}\rangle|$, 
it suffices to prove that
\begin{align*}
\langle\mathcal{D}\rangle\setminus\mathcal{A}=&\left\{(C\setminus\{\ell\})\cup T: C\in\mathcal{C},\ T\in\binom{[\ell+1,n]}{\leq n-|C|+1}\right\},\\
\mathcal{A}\setminus\langle\mathcal{D}\rangle=&\left\{B\cup T: B\in\mathcal{B}\backslash \mathcal{C},\ T\in\binom{[\ell+1,n]}{\leq n-|B|}\right\}.
 \end{align*}

We first prove the description of $\langle\mathcal{D}\rangle\setminus\mathcal{A}$.
For any $F\in\langle\mathcal{D}\rangle\setminus\mathcal{A}$, there exists some $C\in\mathcal{C}$ such that $C\setminus\{\ell\}\subseteq F$.  Since $C\subseteq [\ell]$, we have $C\setminus\{\ell\}\subseteq F\cap[\ell]$.
Moreover, $F\notin\mathcal{A}$  forces  $\ell \notin F\cap[\ell]$; otherwise $C\subseteq F\cap[\ell]$ would imply $F\in\mathcal{A}$, a contradiction.
If $C\setminus\{\ell\}\subsetneq F\cap[\ell]$, then there exists $i\in(F\cap[\ell])\setminus(C\setminus\{\ell\})$ with $i<\ell$. By shiftedness, $(C\setminus\{\ell\})\cup\{i\}\in\mathcal{A}$. Since 
$\mathcal{A}= \langle\mathcal{G}\rangle$, there exists $G\in\mathcal{G}$ such that $G\subseteq(C\setminus\{\ell\})\cup\{i\}\subseteq F\cap[\ell]$, which implies  $F\in\mathcal{A}$, a contradiction. 
Hence, $F\cap[\ell]=C\setminus\{\ell\}$. This implies that
$\langle\mathcal{D}\rangle\setminus\mathcal{A}\subseteq\left\{(C\setminus\{\ell\})\cup T: C\in\mathcal{C},\ T\in\binom{[\ell+1,n]}{\leq n-|C|+1}\right\}.$  
For the reverse, take any  $C\in\mathcal{C}$ and $ T\in\binom{[\ell+1,n]}{\leq n-|C|+1}$.
Since no subset of $C\setminus\{\ell\}$ belongs to $\mathcal{G}$,
 we have $(C\setminus\{\ell\})\cup T \notin \mathcal{A}$. Clearly, $(C\setminus\{\ell\})\cup T \in \langle\mathcal{D}\rangle$. 
Therefore, 
$\left\{(C\setminus\{\ell\})\cup T: C\in\mathcal{C},\ T\in\binom{[\ell+1,n]}{\leq n-|C|+1}\right\} \subseteq\langle\mathcal{D}\rangle\setminus\mathcal{A}$.

We now prove the  description of $\mathcal{A}\setminus\langle\mathcal{D}\rangle$.
For any $A\in\mathcal{A}\setminus\langle\mathcal{D}\rangle$, there exists some $B\in\mathcal{B}\backslash \mathcal{C}$ such that $B\subseteq A$. Then $B\subseteq A\cap[\ell]$.
Observe that $\ell\in B$. 
If $B\subsetneq A\cap[\ell]$, then there exists $i\in(A\cap[\ell])\setminus B$ with $i<\ell$.
By shiftedness, $(B\setminus\{\ell\})\cup\{i\}\in\mathcal{A}$. Since 
$\mathcal{A}= \langle\mathcal{G}\rangle$, there exists $G\in\mathcal{G}$ such that $G\subseteq(B\setminus\{\ell\})\cup\{i\}\subseteq A\cap[\ell]$.
In view of $\ell \notin G$, we have $G\in\mathcal{G}\setminus(\mathcal{B}\cup\mathcal{C})$.
Then $A\in\langle\mathcal{D}\rangle$, a contradiction. Hence, we have $A\cap[\ell]=B$, which implies 
$\mathcal{A}\setminus\langle\mathcal{D}\rangle\subseteq\left\{B\cup T: B\in\mathcal{B}\backslash \mathcal{C},\ T\in\binom{[\ell+1,n]}{\leq n-|B|}\right\}$.
For the reverse inclusion, take any  $B\in\mathcal{B}\backslash \mathcal{C}$ and $ T\in\binom{[\ell+1,n]}{\leq n-|B|}$. Note that no subset of $B$ belongs to $\mathcal{G}\setminus(\mathcal{B}\cup  \mathcal{C})$.
If  $B\cup T \in \langle\mathcal{D}\rangle$, then there exists some  
$C\in\mathcal{C}$ such that $C\setminus\{\ell\}\subseteq B$.
Since $\ell\in B$ and $B\in\mathcal{B}\backslash \mathcal{C}$, we  have $C\subsetneq B$. 
But $\mathcal{B}\cup\mathcal{C}$ is an antichain because $\mathcal{G}$ is a  generating family, a contradiction.
Hence, $B\cup T \notin \langle\mathcal{D}\rangle$, while clearly $B\cup T \in \mathcal{A}$. Consequently,
$\left\{B\cup T: B\in\mathcal{B}\backslash \mathcal{C},\ T\in\binom{[\ell+1,n]}{\leq n-|B|}\right\}\subseteq\mathcal{A}\setminus\langle\mathcal{D}\rangle$.
\end{proof}

For convenience, we introduce the following notation and terminology.

Let $(2^{[n]})^{m} = 2^{[n]} \times \cdots \times 2^{[n]}$. For any $m$-tuple \(\vec{\mathcal{A}} = (\mathcal{A}_1, \ldots, \mathcal{A}_m) \in (2^{[n]})^{m} \),  the $1$-norm of $\vec{\mathcal{A}}$ is defined as  $\|\vec{\mathcal{A}}\| = \sum_{i=1}^{m} |\mathcal{A}_i|$.  We denote $\langle\vec{\mathcal{A}}\rangle=(\langle\mathcal{A}_1\rangle, \ldots, \langle\mathcal{A}_m\rangle),$ where $
\langle\mathcal{A}_i\rangle=\left\{F\subseteq [n]: \exists A \in\mathcal{A}_i\text{ such that }A\subseteq F\right\}.
$
We say that an $m$-tuple \( \vec{\mathcal{A}} \) is \textit{shifted} (\textit{non-empty, monotone, respectively}) if \( \mathcal{A}_i \) is shifted (non-empty, monotone, respectively) for every $i\in [m]$.   Moreover, we say that \( \vec{\mathcal{A}} \) is \textit{  pairwise cross $t$-intersecting} if $\mathcal{A}_1, \ldots, \mathcal{A}_m$ are  pairwise cross $t$-intersecting.
For any $i\in [m]$,
let $\mathcal{G}_i$ be the generating family of $\mathcal{A}_i$.
We say that $\vec{\mathcal{G}} = (\mathcal{G}_1, \ldots, \mathcal{G}_m)$ is the \textit{generating sequence} of $\vec{\mathcal{A}}$. 
The \textit{extent} of $\vec{\mathcal{A}}$ is defined as the maximal extent of $\mathcal{A}_i$ for all $i\in[m]$.

Let $\vec{\mathcal{A}} = (\mathcal{A}_1, \ldots, \mathcal{A}_m)\in (2^{[n]})^{m}$ and $u\in [n]$. 
We define
 $\vec{\mathcal{A}}^{(u)}=(\mathcal{A}_1^{(u)},\ldots, \mathcal{A}_m^{(u)})$ and $\vec{\mathcal{A}}[u]=(\mathcal{A}_1[u],\ldots, \mathcal{A}_m[u])$, and similarly, we define $\vec{\mathcal{A}}(u)=(\mathcal{A}_1(u),\ldots, \mathcal{A}_m(u))$ and $\vec{\mathcal{A}}(\bar{u})=(\mathcal{A}_1(\bar{u}),\ldots, \mathcal{A}_m(\bar{u}))$.
In addition, for any $\vec{\mathcal{A}}, \vec{\mathcal{B}}\in (2^{[n]})^{m}$, we define
$$ \vec{\mathcal{A}} + \vec{\mathcal{B}} = (\mathcal{A}_1 \cup \mathcal{B}_1, \ldots, \mathcal{A}_m \cup \mathcal{B}_m) \quad \text{and} \quad \vec{\mathcal{A}}- \vec{\mathcal{B}} = (\mathcal{A}_1\setminus\mathcal{B}_1, \ldots, \mathcal{A}_m \setminus \mathcal{B}_m). $$

\begin{lemma}\label{le2}
 Let $\vec{\mathcal{A}} = (\mathcal{A}_1, \ldots, \mathcal{A}_m) \in (2^{[n]})^{m}$ be  a monotone shifted pairwise cross $t$-intersecting sequence with  generating sequence  $\vec{\mathcal{G}} = (\mathcal{G}_1, \ldots, \mathcal{G}_m)$. Let $\ell$ be the extent of $\vec{\mathcal{A}}$. If 
$A\in \mathcal{G}_i[\ell]$ and  $B\in \mathcal{G}_j[\ell]$  satisfy $|A\cap B|=t$ for some $i\neq j\in [m]$, then $A\cup B=[\ell]$ and $|A|+|B|=t+\ell$.
\end{lemma}
\begin{proof}
Firstly, we have $A\cup B\subseteq[\ell]$ and $\ell\in A\cap B$. If $A\cup B\subsetneq[\ell]$, then there exists  $x\in[\ell]\setminus(A\cup B)$ and $x<\ell$. By the shiftedness, we get $(A\setminus\{\ell\})\cup\{x\}\in\mathcal{A}_i$ 
and $|((A\setminus\{\ell\})\cup\{x\})\cap B|=|(A\cap B)\setminus\{\ell\}|=t-1$, a contradiction. Then $A\cup B=[\ell]$ and $|A|+|B|=|A\cap B|+|A\cup B|=t+\ell$.
\end{proof}

Recall that $ 
\mathcal{A}_i[u]^{(k)}=\{A\in\mathcal{A}_i: u\in A,|A|=k\}$ 
and
$\vec{\mathcal{A}}[u]^{(k)}=(\mathcal{A}_1[u]^{(k)},\ldots, \mathcal{A}_m[u]^{(k)})$. Inspired by \cite[Lemma 3.1]{L25}, we establish the following result.

\begin{lemma}\label{le22}
  Let $\vec{\mathcal{A}} = (\mathcal{A}_1, \ldots, \mathcal{A}_m) \in (2^{[n]})^{m}$ be  a non-empty monotone shifted pairwise cross $t$-intersecting sequence with  generating sequence  $\vec{\mathcal{G}} = (\mathcal{G}_1, \ldots, \mathcal{G}_m)$. Let $\ell$ be the extent of $\vec{\mathcal{A}}$ with $\ell> t$.
Assume that $\vec{\mathcal{G}}[\ell]^{(t)}=(\emptyset,\ldots,\emptyset)$. Then 
the following statements hold.
\begin{enumerate}
    \item[(1)] If there exist some $u\neq v\in [t,\ell]$ with  $u+v=\ell+t$ such that exactly one of  $\vec{\mathcal{G}}[\ell]^{(u)}$ and $\vec{\mathcal{G}}[\ell]^{(v)}$ is not equal to  $(\emptyset,\ldots, \emptyset)$, then there exists a non-empty monotone pairwise cross $t$-intersecting sequence $\vec{\mathcal{B}}$  with its extent at most $\ell$ such that $\|\vec{\mathcal{B}}\|>\|\vec{\mathcal{A}}\|$.
    
    \item[(2)]  If  $2\nmid \ell+t$, then there exists a non-empty monotone pairwise cross $t$-intersecting sequence $\vec{\mathcal{B}}$  with its extent less than $\ell$  such that $\|\vec{\mathcal{B}}\|\geq \|\vec{\mathcal{A}}\|$.

    \item[(3)]   If  $2\mid \ell+t$ and for all pairs $u\neq v\in [t,\ell]$ with  $u+v=\ell+t$,  either $\vec{\mathcal{G}}[\ell]^{(u)}\neq (\emptyset,\ldots, \emptyset)$ and $\vec{\mathcal{G}}[\ell]^{(v)}\neq (\emptyset,\ldots, \emptyset)$, or $\vec{\mathcal{G}}[\ell]^{(u)}=\vec{\mathcal{G}}[\ell]^{(v)}=(\emptyset,\ldots, \emptyset)$, then there exists a non-empty monotone shifted pairwise cross $t$-intersecting sequence $\vec{\mathcal{B}}$  with its extent at most $\ell$  such that $\|\vec{\mathcal{B}}\|\geq \|\vec{\mathcal{A}}\|$. Moreover, let $\vec{\mathcal{G}}^*$ be the generating sequence of $\vec{\mathcal{B}}$.
If the extent of  $\vec{\mathcal{B}}$ is $\ell$, then  $\vec{\mathcal{G}}^*[\ell]= \vec{\mathcal{G}}[\ell]^{(\frac{\ell+t}{2})}$.
\end{enumerate}
\end{lemma}

\begin{proof}
(1)  By symmetry, we may assume that $\vec{\mathcal{G}}[\ell]^{(u)}\neq (\emptyset,\ldots, \emptyset)$ and $\vec{\mathcal{G}}[\ell]^{(v)}= (\emptyset,\ldots, \emptyset)$.
In view of $\vec{\mathcal{G}}[\ell]^{(t)}=(\emptyset,\ldots,\emptyset)$, we have $u>t$.
Recall that for a family $\mathcal{A}\subseteq 2^{[n]}$ and an integer $\ell \in [n]$, we denote $\mathcal{A}(\ell) = \{A\backslash \{\ell\}: \ell\in A\in \mathcal{A}\}$.
Let  
$$ \vec{\mathcal{G}}'= \vec{\mathcal{G}}- \vec{\mathcal{G}}[\ell]^{(u)}+\vec{\mathcal{G}}[\ell]^{(u)}(\ell),$$
where we write $\vec{\mathcal{G}}[\ell]^{(u)}(\ell)= \left(\mathcal{G}_1[\ell]^{(u)}(\ell),\ldots, \mathcal{G}_m[\ell]^{(u)}(\ell) \right)$. 
By Lemma \ref{le2}, we know that $ \langle\vec{\mathcal{G}}'\rangle$ is a non-empty
pairwise cross $t$-intersecting sequence. In addition, we have $\|\langle\vec{\mathcal{G}}'\rangle\|>\|\langle\vec{\mathcal{G}}\rangle\|=\|\vec{\mathcal{A}}\|$. Setting $\vec{\mathcal{B}}=\langle\vec{\mathcal{G}}'\rangle$ yields the desired result.

(2)   
Since $\vec{\mathcal{G}}[\ell]\neq (\emptyset,\ldots, \emptyset)$  and $2\nmid \ell+t$, it follows  that there exist $u\neq v\in [t,\ell]$ with  $u+v=\ell+t$ such that at least one of $\vec{\mathcal{G}}[\ell]^{(u)}$ and $\vec{\mathcal{G}}[\ell]^{(v)}$ is not equal to  $(\emptyset,\ldots, \emptyset)$. 
For every such pair,  we choose exactly one of the two boundary layers having
the larger norm; in case of equality, choose either one.
Let $\mathcal C$ be the union of all chosen boundary layers.
Then $\|\vec{\mathcal{G}}[\ell]- \mathcal C\|\leq \|\mathcal C\|$ and $ \|\mathcal C\|>0$.
Define 
$$
\vec{\mathcal{G}}^*
=
\vec{\mathcal{G}}-\vec{\mathcal{G}}[\ell]+
\{C\setminus\{\ell\}:C\in\mathcal C\},\qquad \vec{\mathcal{B}}=\langle\vec{\mathcal{G}}^*\rangle.
$$
We also note that every component $\mathcal G_k^{*}$ is non-empty.
Indeed, this is immediate if
$\mathcal G_k\setminus\mathcal G_k[\ell]\ne\emptyset$.
Otherwise every generator in $\mathcal G_k$ contains $\ell$.
Shiftedness then forces $\mathcal G_k=\{[\ell]\}$. Since
$\vec{\mathcal G}[\ell]^{(t)}=(\emptyset,\ldots,\emptyset)$, the
$\ell$-layer is selected from the complementary pair $\{t,\ell\}$.
Consequently, $[\ell-1]\in\mathcal G_k^{*}$.
By  Lemma \ref{le2}, $\vec{\mathcal{B}}$ is pairwise cross $t$-intersecting. 
Morover, Lemma \ref{le1} implies
\begin{align*}
\vec{\mathcal{B}}=\|\langle \vec{\mathcal{G}}^*\rangle\| =\|\vec{\mathcal{A}}\|-\|\vec{\mathcal{G}}[\ell]-\mathcal C\|2^{n-\ell}+\|\mathcal C\|2^{n-\ell}\geq \|\vec{\mathcal{A}}\|.
\end{align*}
Thus, we obtain  a  non-empty pairwise cross $t$-intersecting sequence $\vec{\mathcal{B}}$  with its extent less than $\ell$ and $\|\vec{\mathcal{B}}\|\geq \|\vec{\mathcal{A}}\|$.

(3)   By the given conditions,  for all pairs $u\neq v\in [t,\ell]$ with  $u+v=\ell+t$, either $\vec{\mathcal{G}}[\ell]^{(u)}\neq (\emptyset,\ldots, \emptyset)$ and $\vec{\mathcal{G}}[\ell]^{(v)}\neq (\emptyset,\ldots, \emptyset)$, or $\vec{\mathcal{G}}[\ell]^{(u)}=\vec{\mathcal{G}}[\ell]^{(v)}=(\emptyset,\ldots, \emptyset)$.
Since $\vec{\mathcal{G}}[\ell]\neq (\emptyset,\ldots, \emptyset)$, it follows  that either there exist $u\neq v\in [t,\ell]$ with  $u+v=\ell+t$ such that both $\vec{\mathcal{G}}[\ell]^{(u)}$ and $\vec{\mathcal{G}}[\ell]^{(v)}$ are  not equal to  $(\emptyset,\ldots, \emptyset)$, or  $2\mid \ell+t$ and $\vec{\mathcal{G}}[\ell]= \vec{\mathcal{G}}[\ell]^{(\frac{\ell+t}{2})}$. 
If the latter holds, then we are done. We now turn to the former.
From $\vec{\mathcal{G}}[\ell]^{(t)}=(\emptyset,\ldots,\emptyset)$, we obtain $t<u, v<\ell$. 
Combining this with the shiftedness, we obtain that
$\mathcal{G}_k\backslash \mathcal{G}_k[\ell]\neq \emptyset$  for all $k\in[m]$.
For every pair
$u\neq v\in [t,\ell]$ with $u+v=\ell+t$ such that both
$\vec{\mathcal{G}}[\ell]^{(u)}$ and $\vec{\mathcal{G}}[\ell]^{(v)}$ are not equal to
$(\emptyset,\ldots,\emptyset)$,  we choose exactly one of the two boundary layers having
the larger norm; in case of equality, choose either one.
Let $\mathcal C$ be the union of all chosen boundary layers.
Then $\|\vec{\mathcal{G}}[\ell]-\vec{\mathcal{G}}[\ell]^{(\frac{\ell+t}{2})}- \mathcal C\|\leq \|\mathcal C\|$.
Define 
$$
\vec{\mathcal{G}}^*
=
\vec{\mathcal{G}}-\vec{\mathcal{G}}[\ell]+\vec{\mathcal{G}}[\ell]^{(\frac{\ell+t}{2})}+
\{C\setminus\{\ell\}:C\in\mathcal C\},\qquad \vec{\mathcal{B}}=\langle\vec{\mathcal{G}}^*\rangle.
$$

We now verify that $\vec{\mathcal B}$ is shifted. Fix $k\in[m]$
and $1\leq i<j\leq n$. Let $F\in\mathcal B_k$ satisfy
$j\in F$ and $i\notin F$, and put
$
  F'=(F\setminus\{j\})\cup\{i\}.
$
Choose $D\in\mathcal G_k^{*}$ such that $D\subseteq F$. If
$j\notin D$, then $D\subseteq F'$, and hence $F'\in\mathcal B_k$.
We may therefore assume that $j\in D$.
First, suppose that
$
  D\in\mathcal G_k\setminus\mathcal G_k[\ell].
$
Then $\ell\notin D$. Since $\mathcal A_k$ is shifted,
$
  D'=(D\setminus\{j\})\cup\{i\}\in\mathcal A_k.
$
Hence, there exists $E\in\mathcal G_k$ with $E\subseteq D'$. Since
$\ell\notin D'$, we also have $\ell\notin E$, so $E$ is retained in
$\mathcal G_k^{*}$. Thus $E\subseteq F'$ and
$F'\in\mathcal B_k$.
Next, suppose that
$
  D\in\mathcal G_k[\ell]^{(\frac{\ell+t}{2}}.
$
If $j=\ell$, then $D'=(D\setminus\{\ell\})\cup\{i\}$ contains no
$\ell$, and the preceding argument applies. Suppose that $j<\ell$.
Choose $E\in\mathcal G_k$ with $E\subseteq D'$. If $\ell\notin E$,
then $E$ is retained. If $\ell\in E$ and
$|E|=\frac{\ell+t}{2}$, then $E=D'$ and $E$ belongs to the retained middle
boundary layer. Finally, if $\ell\in E$ and
$|E|<\frac{\ell+t}{2}$, choose
$
  h\in D'\setminus E.
$
Then $h<\ell$, and shiftedness gives
$
  (E\setminus\{\ell\})\cup\{h\}\in\mathcal A_k.
$
This set is contained in $D'$ and contains no $\ell$, so it contains
some non-boundary generator belonging to
$\mathcal G_k\setminus\mathcal G_k[\ell]$, which is retained in
$\mathcal G_k^{*}$. Hence again $F'\in\mathcal B_k$.
Finally, suppose that
$
  D=C\setminus\{\ell\}
$
for some generator $C$. Since
$j\in D$ and $i\notin D$, shiftedness of $\mathcal A_k$ gives
$
  C'=(C\setminus\{j\})\cup\{i\}\in\mathcal A_k.
$
Choose $E\in\mathcal G_k$ with $E\subseteq C'$. If $\ell\notin E$,
then $E\subseteq D'$ and $E$ is retained, where
$
  D'=(D\setminus\{j\})\cup\{i\}.
$
If $\ell\in E$ and $|E|=|C|$, then $E=C'$. Since the entire
$|C|$-boundary layer was selected, we have
$
  E\setminus\{\ell\}=D'\in\mathcal G_k^{*}.
$
If $\ell\in E$ and $|E|<|C|$, choose
$
  h\in D'\setminus(E\setminus\{\ell\}).
$
Then
$
  (E\setminus\{\ell\})\cup\{h\}\in\mathcal A_k
$
by shiftedness. This is a subset of $D'$ not containing $\ell$, and
therefore it contains a retained non-boundary generator. Thus
$F'\in\mathcal B_k$ in all cases.

We conclude that each $\mathcal B_k$ is shifted, and hence
$\vec{\mathcal B}$ is a shifted sequence.
By  Lemma \ref{le2}, $\vec{\mathcal{B}}$ is also pairwise cross $t$-intersecting. 
Morover, Lemma \ref{le1} implies
\begin{align*}
\vec{\mathcal{B}}=\|\langle \vec{\mathcal{G}}^*\rangle\| =\|\vec{\mathcal{A}}\|-\|\vec{\mathcal{G}}[\ell]-\vec{\mathcal{G}}[\ell]^{(\frac{\ell+t}{2})}- \mathcal C\|2^{n-\ell}+\|\mathcal C\|2^{n-\ell}\geq \|\vec{\mathcal{A}}\|.
\end{align*}
Thus, we  obtain a non-empty monotone shifted pairwise cross $t$-intersecting sequence $\vec{\mathcal{B}}$  with its extent at most $\ell$ and $\|\vec{\mathcal{B}}\|\geq \|\vec{\mathcal{A}}\|$. 
After deleting redundant members from $\vec{\mathcal G}^{*}$ and keeping the same notation, every $M\in\mathcal G_k[\ell]^{(\frac{\ell+t}{2})}$ remains minimal, since a retained non-boundary generator contained in $M$ would contradict the antichain property of $\mathcal G_k$, while $C\setminus\{\ell\}\subseteq M$ would imply $C\subseteq M$, again contradicting the antichain property of
$\mathcal G_k$. As all other members avoid $\ell$, we obtain $\vec{\mathcal{G}}^*[\ell]= \vec{\mathcal{G}}[\ell]^{(\frac{\ell+t}{2})}$ whenever $\vec{\mathcal B}$ has extent $\ell$.
\end{proof}

\subsection{Pushing-pulling method}
The pushing-pulling method was  introduced by  Ahlswede and Khachatrian \cite{AK1999}.
For a set $A \subseteq [n]$ and integers $1 \leq i\neq j \leq n$, 
let $A\Delta \{i,j\}$ denote the symmetric difference of $A$ and $ \{i,j\}$.
Let 
$A_{i,j}$ denote the $(i,j)$-\textit{exchange operator}, where
$$
A_{i,j}= \begin{cases}A\Delta \{i,j\} & \text { if } |A\cap \{i,j\}|=1; \\ A & \text { otherwise.}\end{cases}
$$ 
Clearly, we have $|A_{i,j}|=|A|$.
Let $\mathcal{A} \subseteq 2^{[n]}$ be a family. Define  $\mathcal{A}_{i,j}=\{A_{i,j}: A\in \mathcal{A}\}$. We say that $\mathcal{A}$ is \textit{exchange stable} on $[s]$ if $
\mathcal{A} = \mathcal{A}_{i,j} \text{ for all } 1 \leq i\neq j \leq s$.
Note that in this case if $\{A\cap [s]: A \in \mathcal{A}\}$ contains an $a$-element subset, then it contains all $a$-element subsets of $[s]$.
The \textit{symmetric extent} of $\mathcal{A}$ is the largest integer $s$ such that $\mathcal{A}$ is  exchange stable on $[s]$, i.e.,
$$
\mathcal{A} = \mathcal{A}_{i,j} \text{ for all } 1 \leq i\neq j \leq s, \text{ but } \mathcal{A} \neq \mathcal{A}_{i,s+1} \text{ for some } 1 \leq i \leq s.
$$
{ For any $r \in [n-1]$}, we define
$$
\mathcal{A}^{-r} = \{ A \in \mathcal{A}: A_{i,r+1} \notin \mathcal{A} \text{ for some } 1 \leq i \leq r \}.
$$
If $\mathcal{A}\neq \emptyset$ with its symmetric extent $s<n$,  then $\mathcal{A}^{-s}\neq \emptyset$ and $\mathcal{A}^{-r}=\emptyset$ for all $r<s$.

 For \(\vec{\mathcal{A}} = (\mathcal{A}_1, \ldots, \mathcal{A}_m) \in (2^{[n]})^{m} \), the \textit{symmetric extent} of $\vec{\mathcal{A}}$ is defined as the minimal symmetric extent of $\mathcal{A}_i$ for all $i\in[m]$. 
We need the following crucial properties concerning the symmetric extent, which is inspired by \cite[Lemma 2]{AK1999}.

\begin{lemma}\label{le3}
Let $\vec{\mathcal{A}} = (\mathcal{A}_1, \ldots, \mathcal{A}_m) \in (2^{[n]})^{m}$ be  a shifted  sequence. Let $s$ be the symmetric extent of $\vec{\mathcal{A}}$ with $s<n$. 
Then the following statements hold.
\begin{enumerate}
    \item[(1)] If  $\mathcal{A}_k^{-s}\neq \emptyset$, then $s + 1 \notin A$ for any $A \in \mathcal{A}_k^{-s}$.
    
    \item[(2)] If  $\mathcal{A}_k^{-s}\neq \emptyset$, then  $A_{j,s+1} \notin \mathcal{A}_k$ for any $A \in \mathcal{A}_k^{-s}$ and $j \in A\cap [s]$.
    
    \item[(3)] If  $\mathcal{A}_k^{-s}\neq \emptyset$ and $A= B \cup C \in \mathcal{A}_k^{-s}$, where $B = A \cap [1,s]$ and $C = A \cap [s+ 1,n]$,
then $B' \cup C \in \mathcal{A}_k^{-s}$ for every $B' \subseteq [s]$ with $|B'| = |B|$.
    
    \item[(4)]  Assume that $\vec{\mathcal{A}}$ is pairwise cross $t$-intersecting,  $\mathcal{A}_k^{-s}\neq \emptyset$ and $ \mathcal{A}_q \setminus \mathcal{A}_q^{-s}\neq \emptyset$ for some $k\neq q \in [m]$. Let $A \in \mathcal{A}_k^{-s}$ and $D \in \mathcal{A}_q \setminus \mathcal{A}_q^{-s}$. Then
    \[ |A_{i,s+1} \cap D| \geq  t \text{ for all } 1 \leq i \leq s. \]

    \item[(5)]  Assume that $\vec{\mathcal{A}}$ is pairwise cross $t$-intersecting,  $\mathcal{A}_k^{-s}\neq \emptyset$ and $ \mathcal{A}_q^{-s}\neq \emptyset$ for some $k\neq q \in [m]$. Let $A_1\in \mathcal{A}_k^{-s}$ and $A_2\in \mathcal{A}_q^{-s}$ be such that $|A_1\cap [s]|+|A_2\cap [s]|\neq s+t$. Then $|A_1 \cap A_2| \geq  t + 1$.
\end{enumerate}

\end{lemma} 
\begin{proof}
(1) For any $A \in \mathcal{A}_k^{-s}$, there exists some $i\in [s]$
such that $
A_{i,s+1} \notin \mathcal{A}_k$. This implies that $|A\cap \{i,s+1\}|=1$. If 
$A\cap \{i,s+1\}=\{s+1\}$, then $A_{i,s+1}=(A \backslash\{s+1\}) \cup\{i\} $. Note that $ \mathcal{A}_k$ is shifted.
Thus, we have $A_{i,s+1}\in  \mathcal{A}_k$, which is a contradiction. This proves that 
$A\cap \{i,s+1\}=\{i\}$ and $s+1\notin A$.

(2) For any $A \in \mathcal{A}_k^{-s}$, there exists some $i\in [s]$
such that $
A_{i,s+1} \notin \mathcal{A}_k$. By (1), we get $A\cap \{i,s+1\}=\{i\}$.
If $j=i$, then we are done. So  let $j\neq i$.
Suppose for the contradiction that $A_{j,s+1} \in \mathcal{A}_k$. 
Since the symmetric extent of $\mathcal{A}_k$ is at least $s$, we infer that $(A_{j,s+1})_{i,j}\in \mathcal{A}_k$.
As $A_{j,s+1}=(A \backslash\{j\}) \cup\{s+1\} $, we have
$(A_{j,s+1})_{i,j}=(A \backslash\{i\}) \cup\{s+1\}=A_{i,s+1}$. However, we have $ A_{i,s+1}\notin \mathcal{A}_k$, which leads to a contradiction. Thus, we conclude that 
$A_{j,s+1} \notin \mathcal{A}_k$.

(3) Since the symmetric extent of $\mathcal{A}_k$ is at least $s$, we know by definition that $\mathcal{A}_k$ is exchange stable on $[s]$. This leads to 
$B' \cup C \in \mathcal{A}_k$. Assume for the sake of contradiction that $B' \cup C\notin \mathcal{A}_k^{-s}$. Then $(B' \cup C)_{i,s+1} \in \mathcal{A}_k$ for all $i\in[s]$.
By (1), we get $s+1\notin C$.
We can choose $j'\in B'$ and $j\in B$ arbitrarily. Then we have
$(B' \cup C)_{j', s+1}=(B' \backslash\{j'\}) \cup C\cup\{s+1\}\in \mathcal{A}_k$ and 
$(B \cup C)_{j, s+1}=(B \backslash\{j\}) \cup C\cup\{s+1\}$. 
Observe that $(B \backslash\{j\}) \cup C\cup\{s+1\}$  can be  obtained from
 $(B' \backslash\{j'\}) \cup C\cup\{s+1\}$
 by applying exchange operations on $[s]$.
By the exchange stability, we derive that $(B \backslash\{j\}) \cup C\cup\{s+1\}\in \mathcal{A}_k$.
That is  ${ A_{j, s+1}\in \mathcal{A}_k}$.
However, (2) leads to $A_{j, s+1}\notin \mathcal{A}_k$, a contradiction.
Therefore, we conclude that $B' \cup C\in \mathcal{A}_k^{-s}$.

(4) By applying (1), we have $s + 1 \notin A$. We choose $i\in [s]$ arbitrarily.
If $i\notin A$, then $A_{i,s+1}=A $. Since $\mathcal{A}_k$ and $\mathcal{A}_q$ are cross $t$-intersecting, we get $ |A_{i,s+1} \cap D|=|A\cap D| \geq  t$.
Next, we assume that $i\in A$.
Then
$A_{i,s+1}=(A \backslash\{i\}) \cup\{s+1\}$. Observe that $i\notin D$ or $s+1\in D$ would imply $ |A_{i,s+1} \cap D|\geq  |A\cap D| \geq  t$.
So  let  $i\in D$ and $s+1\notin D$.
It follows from $D \in \mathcal{A}_q \setminus \mathcal{A}_q^{-s}$
that $D_{i,s+1}\in\mathcal{A}_q $. This leads to  $D_{i,s+1}=(D \backslash\{i\}) \cup\{s+1\}\in\mathcal{A}_q $. Invoking the fact that $\mathcal{A}_k$ and $\mathcal{A}_q$ are cross $t$-intersecting, we obtain $ |A_{i,s+1} \cap D|= |A\cap D_{i,s+1}| \geq  t$.

(5) If $|A_1\cap [s]|+|A_2\cap [s]|\geq  s+t+1$, then $|A_1\cap A_2|\geq  |A_1\cap A_2\cap [s]|\geq  |A_1\cap [s]|+|A_2\cap [s]|-s\geq  t+1$, as desired. Next, let $|A_1\cap [s]|+|A_2\cap [s]|\leq s+t-1$. For convenience, we denote $a:=|A_1\cap [s]|$ and $b:=|A_2\cap [s]|$.
Since $\mathcal{A}_k$  and $\mathcal{A}_q$ are both exchange stable on $[s]$, we infer that $A_1':=[a]\cup (A_1\cap[s+1, n])\in \mathcal{A}_k$
and  $A_2':=[s-b+1, s]\cup (A_2\cap[s+1, n])\in \mathcal{A}_q$.
Then $|A_1'\cap A_2'\cap [s]|=|[a]\cap [s-b+1,s]|\leq \text{max}\{0, a+b-s\}\leq t-1$. This implies that $|A_1'\cap A_2'\cap [s+1,n]|\geq  1$ because $\mathcal{A}_k$  and $\mathcal{A}_q$ are cross $t$-intersecting. 
Thus, we have $|A_1\cap A_2\cap [s+1,n]|\geq  1$. Using (1), we get $s+1\notin  A_1\cap A_2$. Consequently, there exists $x\in A_1\cap A_2\cap [s+2,n]$.
By the shiftedness, $(A_1\backslash\{x\}) \cup\{s+1\}\in \mathcal{A}_k$, which together with the cross $t$-intersecting property, yields
$|A_1\cap A_2|=|((A_1\backslash\{x\}) \cup\{s+1\})\cap A_2|+1\geq  t+1$, as needed. 
\end{proof}

Following the previous notation, we define $\mathcal{A}_k^{-s} = \{ A \in \mathcal{A}_k: A_{i,s+1} \notin \mathcal{A}_k \text{ for some } 1 \leq i \leq s \}
$ and we denote 
$\vec{\mathcal{A}}^{-s} = (\mathcal{A}_1^{-s}, \ldots, \mathcal{A}_m^{-s})$. 
For every $\mathcal{A}_k$, consider the following partition:
$$
\mathcal{A}_{k}^{-s}=\bigcup_{i=1}^{s}\mathcal{A}_{k}^{-s, i}, 
$$
where $\mathcal{A}_{k}^{-s,i}=\{A\in\mathcal{A}_{k}^{-s}: |A\cap[s]|=i \}$. 
Let $\mathcal{A}_{k,*}^{-s, i}=\{A\cap [s+2, n]: A\in \mathcal{A}_{k}^{-s,i}\}$.  Moreover, let
 $\vec{\mathcal{A}}^{-s,i} = (\mathcal{A}_1^{-s,i}, \ldots, \mathcal{A}_m^{-s,i})$ and $\vec{\mathcal{A}}_*^{-s,i} = (\mathcal{A}_{1,*}^{-s,i}, \ldots, \mathcal{A}_{m,*}^{-s,i})$.

To apply the pushing-pulling method, we need the following result.

\begin{lemma}\label{le32}
Let $\mathcal A_k\subseteq 2^{[n]}$ be a monotone
shifted family. Let $\ell$ and $s$ be the extent and the symmetric
extent of $\mathcal A_k$, respectively. Suppose that $\ell<n$ and
$\mathcal A_{k,*}^{-s,i}\ne\emptyset$. Then, for any $x\in[n]$
satisfying
$
x>\ell$ and $ x>s+1$,
we have
\[ \left|\mathcal{A}_{k,*}^{-s,i}[x] \right|= 
\left|\mathcal{A}_{k,*}^{-s,i}(\bar{x}) \right|. \]
\end{lemma} 

\begin{proof}
Since $x>s+1$, the coordinate $x$ belongs to $[s+2,n]$, while
$x>\ell$ guarantees that membership in $\mathcal A_k$ is unchanged
by adding or deleting $x$.
We are going to establish a bijection from $ \mathcal{A}_{k,*}^{-s,i}[x]$ to 
$\mathcal{A}_{k,*}^{-s,i}(\bar{x})$.
Let $\phi: \mathcal{A}_{k,*}^{-s,i}[x] \longrightarrow\mathcal{A}_{k,*}^{-s,i}(\bar{x})$ be a map defined as $\phi(C)=C\backslash\{x\}$, where $C\in \mathcal{A}_{k,*}^{-s,i}[x]$. 
It is sufficient to prove that $\phi$ is a bijection.  Consequently, we get $|\mathcal{A}_{k,*}^{-s,i}[x]|=|\mathcal{A}_{k,*}^{-s,i}(\bar{x})|$. 

To start with, we choose $C\in \mathcal{A}_{k,*}^{-s,i}[x]$  arbitrarily. Then there exists a set $A\in \mathcal{A}_k^{-s,i}$ such that $A\cap [s+2, n]=C$.
By Part (1) of Lemma \ref{le3}, we have $s+1\notin A$.
Since $\ell$ is the extent of  $\mathcal{A}_k$ and $\ell<x$, we have $A\backslash\{x\}\in \mathcal{A}_k$. We claim that $A\backslash\{x\}\in \mathcal{A}_k^{-s, i}$. It suffices to show that $A\backslash\{x\}\in \mathcal{A}_k^{-s}$.
Suppose on the contrary that $A\backslash\{x\}\in \mathcal{A}_k\backslash\mathcal{A}_k^{-s}$. Then for any $a\in A\cap [s]$, we have 
$(A\backslash\{a, x\})\cup\{s+1\} \in \mathcal{A}_k$. Since $\mathcal{A}_k$ is monotone, we infer that $(A\backslash\{a\})\cup\{s+1\}\in \mathcal{A}_k$. It follows that
$A_{a,s+1} \in \mathcal{A}_k$ for all $a\in [s]$, which a contradiction.
Now, we must have $A\backslash\{x\}\in \mathcal{A}_k^{-s, i}$, which implies that $C\backslash\{x\}\in \mathcal{A}_{k,*}^{-s, i}(\bar{x})$.

On the other hand, we choose $C'\in \mathcal{A}_{k,*}^{-s,i}(\bar{x})$  arbitrarily.
Then there exists some $A'\in \mathcal{A}_k^{-s,i}$ such that $A'\cap [s+2, n]=C'$. By Part (1) of Lemma \ref{le3}, we have $s+1\notin A'$. Since $\mathcal{A}_k$ is monotone, we have $A'\cup\{x\}\in \mathcal{A}_k$. We claim that $A'\cup\{x\}\in \mathcal{A}_k^{-s, i}$.
Suppose on the contrary that $A'\cup\{x\}\in \mathcal{A}_k\backslash\mathcal{A}_k^{-s}$. Then for any $b\in A'\cap [s]$, we have 
$(A'\backslash\{b\})\cup\{s+1,x\} \in \mathcal{A}_k$.
Since $\ell$ is the extent of  $\mathcal{A}_k$ and  $\ell<x$, we have $(A'\backslash\{b\})\cup\{s+1\} \in \mathcal{A}_k$.
Then
$A'_{b,s+1} \in \mathcal{A}_k$ for all $b\in [s]$, a contradiction.
Therefore, we have $A'\cup\{x\}\in \mathcal{A}_k^{-s, i}$. So we get 
$C'\cup\{x\}\in \mathcal{A}_{k,*}^{-s, i}[x]$. 
Thus, we conclude that $\phi$ is indeed a bijection from $ \mathcal{A}_{k,*}^{-s,i}[x]$ to 
$\mathcal{A}_{k,*}^{-s,i}(\bar{x})$.
\end{proof}

\begin{lemma}\label{le33}
Let $\vec{\mathcal{A}} = (\mathcal{A}_1, \ldots, \mathcal{A}_m) \in (2^{[n]})^{m}$ be  a non-empty  monotone shifted  sequence. Let $\ell$ and   $s$ be  the extent and the symmetric extent of $\vec{\mathcal{A}}$, respectively. Suppose that $s<\ell<n$ and $\vec{\mathcal{A}}_*^{-s,i}\neq(\emptyset, \ldots, \emptyset)$.  Then, for any $x\in [\ell+1,n]$,   $\|\vec{\mathcal{A}}_*^{-s,i}[x]\|=\|\vec{\mathcal{A}}_*^{-s,i}(\bar{x})\|$.
\end{lemma} 
\begin{proof}
Let $K=\{k: \mathcal{A}_{k,*}^{-s, i}\neq \emptyset, k\in[m]\}$. Then $K\neq \emptyset$ because $\vec{\mathcal{A}}_*^{-s,i}\neq(\emptyset, \ldots, \emptyset)$. 
For each $k\in K$, let $\ell_k$ be the extent of $\mathcal A_k$.
Then $\ell_k\le\ell$. 
Since $\mathcal A_{k,*}^{-s,i}\ne\emptyset$, we have
$\mathcal A_k^{-s}\ne\emptyset$. As the symmetric extent of
$\mathcal A_k$ is at least $s$, it is therefore exactly $s$.
Since $x\in[\ell+1,n]$ and $s<\ell$, we have
$x>\ell_k$ and $x>s+1$. Hence, Lemma \ref{le32} gives
\[
\bigl|\mathcal A_{k,*}^{-s,i}[x]\bigr|
=
\bigl|\mathcal A_{k,*}^{-s,i}(\bar x)\bigr|.
\]
Consequently, we get $\|\vec{\mathcal{A}}_*^{-s,i}[x]\|=\sum_{k\in K}|\mathcal{A}_{k, *}^{-s,i}[x]|=\sum_{k\in K}|\mathcal{A}_{k, *}^{-s,i}(\bar{x})|=\|\vec{\mathcal{A}}_*^{-s,i}(\bar{x})\|$.
\end{proof}

The following lemma is the key ingredient in our proof.

\begin{lemma}\label{le5}
Let $\vec{\mathcal{A}} = (\mathcal{A}_1, \ldots, \mathcal{A}_m) \in (2^{[n]})^{m}$ be  a non-empty monotone shifted pairwise cross $t$-intersecting  sequence. Let $\ell$ and   $s$ be  the extent and the symmetric extent of $\vec{\mathcal{A}}$, respectively. Assume that  $\mathcal{A}_{k}^{-s, i}=\emptyset$ for all $i\in [t-1]$ and $k\in [m]$. If $n>\ell>s$, $t>1$ and $2|(\ell+t)$,   then there exists a non-empty monotone pairwise cross $t$-intersecting sequence $\vec{\mathcal{B}}$  with  extent at most $\ell$ such that $\|\vec{\mathcal{B}}\|>\|\vec{\mathcal{A}}\|$.
\end{lemma} 
\begin{proof}
Since $s<n$, by the definition of $s$, we have 
$\vec{\mathcal{A}}^{-s} = (\mathcal{A}_1^{-s}, \ldots, \mathcal{A}_m^{-s})\neq (\emptyset,\ldots, \emptyset).$ Since  $\vec{\mathcal{A}}$ is monotone, we infer that $[n]\in \mathcal{A}_k$ for all $k\in [m]$. Clearly, we have $[n]\notin \mathcal{A}_k^{-s}$.
It follows that $\mathcal{A}_k\backslash \mathcal{A}_k^{-s}\neq \emptyset$ for all $k\in [m]$.

 Since $\mathcal{A}_{k}^{-s, i}=\emptyset$ for all $i\in [t-1]$ and $k\in [m]$, and $\vec{\mathcal{A}}^{-s} = (\mathcal{A}_1^{-s}, \ldots, \mathcal{A}_m^{-s})\neq (\emptyset,\ldots, \emptyset)$, 
 there exists some $a\in [t,s]$ such that 
$\vec{\mathcal{A}}^{-s,a} = (\mathcal{A}_1^{-s,a}, \ldots, \mathcal{A}_m^{-s,a})\neq (\emptyset,\ldots, \emptyset)$.
By Part (3) of Lemma \ref{le3}, we know that if $\mathcal{A}_{k}^{-s, a}\neq \emptyset$, then
$|\mathcal{A}_{k}^{-s, a}|=\binom{s}{a}|\mathcal{A}_{k,*}^{-s, a}|$. 
Thus, we have
$\|\vec{\mathcal{A}}^{-s,a}\|=\binom{s}{a}\|\vec{\mathcal{A}}_*^{-s,a}\|$.

For every $k\in [m]$ and $i\in [t,s]$, we define
$$
\mathcal{B}(k,s, i)=
 \begin{cases} \{B: B\cap [s]=i-1, s+1\in B, B\cap [s+2,n]\in \mathcal{A}_{k,*}^{-s, i} \} & \text { if }  \mathcal{A}_{k}^{-s, i}\neq \emptyset; \\ \emptyset & \text { otherwise. }\end{cases}
$$
Let $\vec{\mathcal{B}}(s,i) = (\mathcal{B}(1,s, i), \ldots, \mathcal{B}(m,s, i))$.
It follows from Part (2) of Lemma \ref{le3} 
that $\mathcal{B}(k,s, i)\cap \mathcal{A}_{k}=\emptyset$ for all $k\in[m]$ if $\mathcal{B}(k,s, i)\neq \emptyset$. Moreover, we have that $\vec{\mathcal{B}}(s,a)\neq (\emptyset,\ldots, \emptyset)$ and $\|\vec{\mathcal{B}}(s,a)\|=\binom{s}{a-1}\|\vec{\mathcal{A}}_*^{-s,a}\|$ since $\vec{\mathcal{A}}^{-s,a} \neq (\emptyset,\ldots, \emptyset)$. 

We have the following two cases $a\neq s+t-a$; or $a=\frac{s+t}{2}$.

{\bf Case 1.}~For the case $a\neq s+t-a$, let 
\begin{align*}
 \vec{\mathcal{A}}' = \vec{\mathcal{A}}- \vec{\mathcal{A}}^{-s,a}+\vec{\mathcal{B}}(s,s+t-a)
\end{align*}
and 
\[ \vec{\mathcal{A}}'' =\vec{\mathcal{A}}- \vec{\mathcal{A}}^{-s,s+t-a}+\vec{\mathcal{B}}(s,a). \]
Recall that $\mathcal{A}_k\backslash \mathcal{A}_k^{-s}\neq \emptyset$ for all $k\in [m]$.

Using Parts (4) and (5) of Lemma \ref{le3}, we infer that both $ \vec{\mathcal{A}}'$  and $ \vec{\mathcal{A}}''$ are  
non-empty pairwise cross $t$-intersecting sequences.
Since $\vec{\mathcal{B}}(s,a)\neq (\emptyset,\ldots, \emptyset)$, we have $\|\vec{\mathcal{B}}(s,a)\|>0$. Recall that $\mathcal{B}(k,s, i)\cap \mathcal{A}_{k}=\emptyset$ for all $k\in[m]$ and $\mathcal{B}(k,s, i)\neq \emptyset$.

If $\vec{\mathcal{A}}^{-s,s+t-a}= (\emptyset,\ldots, \emptyset)$, then $\|\vec{\mathcal{A}}''\|=\|\vec{\mathcal{A}}\|+\|\vec{\mathcal{B}}(s,a)\|>\|\vec{\mathcal{A}}\|$. Setting $\vec{\mathcal{B}}=\vec{\mathcal{A}}''$, we get the { desired result.}   
If $\vec{\mathcal{A}}^{-s,s+t-a}\neq (\emptyset,\ldots, \emptyset)$, then $\|\vec{\mathcal{A}}^{-s,s+t-a}\|
=\binom{s}{s+t-a}\|\vec{\mathcal{A}}_*^{-s,s+t-a}\|$.
Moreover, we have $\vec{\mathcal{B}}(s,s+t-a)\neq (\emptyset,\ldots, \emptyset)$ and $\|\vec{\mathcal{B}}(s,s+t-a)\|=\binom{s}{s+t-a-1}\|\vec{\mathcal{A}}_*^{-s,s+t-a}\|$.
\begin{claim}\label{cl1}
$\max \left\{\|\vec{\mathcal{A}}'\|, \|\vec{\mathcal{A}}''\|\right\}>\|\vec{\mathcal{A}}\|$.
\end{claim}

\begin{proof}[Proof of Claim \ref{cl1}]
Suppose on the contrary that $\|\vec{\mathcal{A}}'\|\leq\|\vec{\mathcal{A}}\|$
and
$\|\vec{\mathcal{A}}''\|\leq\|\vec{\mathcal{A}}\|$.
It follows that $\|\vec{\mathcal{A}}^{-s,a}\|\geq  \|\vec{\mathcal{B}}(s,s+t-a)\|$ and $\|\vec{\mathcal{A}}^{-s,s+t-a}\|\geq  \|\vec{\mathcal{B}}(s,a)\|$.
 Consequently, we get 
\begin{align*}
 \binom{s}{a}\|\vec{\mathcal{A}}_*^{-s,a}\| 
 \geq  \binom{s}{s+t-a-1}\|\vec{\mathcal{A}}_*^{-s,s+t-a}\|>0
\end{align*}
and 
\[ \binom{s}{s+t-a}\|\vec{\mathcal{A}}_*^{-s,s+t-a}\| \geq  \binom{s}{a-1}\|\vec{\mathcal{A}}_*^{-s,a}\|>0. \]
Multiplying these two inequalities, we obtain $\binom{s}{a}\binom{s}{s+t-a}\geq  \binom{s}{s+t-a-1} \binom{s}{a-1}$. This implies that $a(s+t-a)\leq (s-a+1)(a+1-t)$, { yielding a contradiction since $t>1$.} 
\end{proof} 
In what follows, we denote  
$$
\vec{\mathcal{B}}= \begin{cases} \langle\vec{\mathcal{A}}' \rangle& \text { if } \|\vec{\mathcal{A}}'\|\geq \|\vec{\mathcal{A}}''\|;\\ 
\langle\vec{\mathcal{A}}''\rangle & \text { otherwise. }\end{cases}
$$
It remains to check the extent. For every $y>\ell$, membership in each
$\mathcal A_k$ is unchanged by adding or deleting $y$. Since the
definition of $\mathcal A_k^{-s,i}$ only uses exchanges between points
of $[s]$ and $s+1$, the same is true of
$\mathcal A_k^{-s,i}$ and $\mathcal B(k,s,i)$. Thus,  both $\langle\vec{\mathcal A}'\rangle$ and
$\langle\vec{\mathcal A}''\rangle$ are independent of every coordinate
$y>\ell$, so no generating set of the chosen sequence contains such a
coordinate. Hence, $\vec{\mathcal B}$ has extent at most $\ell$, and the
result follows from  Claim \ref{cl1}.

{\bf Case 2.}~For the case $a=\frac{s+t}{2}$, by the previous case, we may assume that 
$\vec{\mathcal{A}}^{-s,j} =(\emptyset,\ldots, \emptyset)$ for all $j\in [t,s]$ and $j\neq s+t-j$.
Since $s<\ell$ and $2|(\ell+t)$, we have $s\leq \ell-2< n-2$.
From $\vec{\mathcal{A}}^{-s,\frac{s+t}{2}} \neq(\emptyset,\ldots, \emptyset)$, we know that $\vec{\mathcal{A}}_*^{\,-s,\frac{s+t}{2}}=\Big(\mathcal{A}_{1,*}^{-s,\frac{s+t}{2}}, \ldots, \mathcal{A}_{m,*}^{-s,\frac{s+t}{2}}\Big) \neq(\emptyset,\ldots, \emptyset)$. 

Recall that for a family $\mathcal{A}\subseteq 2^{[n]}$ and $x\in[\ell+1, n]$, we denote $\mathcal{A}[x]=\{A\in \mathcal{A}: x\in A\}$ and $\mathcal{A}(\bar{x})=\{A\in \mathcal{A}: x\notin A\}$.
Consider the sequences
$\vec{\mathcal{A}}^{-s,\frac{s+t}{2}}(\bar{x}) = \Big(\mathcal{A}_1^{-s,\frac{s+t}{2}}(\bar{x}), \ldots, \mathcal{A}_m^{-s,\frac{s+t}{2}}(\bar{x})\Big)$ and  $\vec{\mathcal{B}}(s,\frac{s+t}{2})[x] = \left(\mathcal{B}(1,s, \frac{s+t}{2})[x], \ldots, \mathcal{B}(m,s, \frac{s+t}{2})[x]\right)$.
Let 
\begin{align*}
 \vec{\mathcal{A}}'''&= \vec{\mathcal{A}}- \vec{\mathcal{A}}^{-s,\tfrac{s+t}{2}}(\bar{x})+\vec{\mathcal{B}}(s,\tfrac{s+t}{2})[x].
\end{align*}
 { For any $A\in \mathcal{A}_i^{-s,\frac{s+t}{2}}[x]$
and $B\in  \mathcal{B}(j,s, \frac{s+t}{2})[x]$ with $i\neq j$, we have $|A\cap [s]|= \frac{s+t}{2}$ and $|B\cap [s]|=\frac{s+t}{2}-1$ and $x\in A\cap B$.
Consequently, $|A\cap B|\geq |A\cap B\cap [s]|+1\geq |A\cap [s]|+|B\cap [s]|-s+1=t$.}
Combining this with Parts (4) and (5) of Lemma \ref{le3},  we obtain that 
$\vec{\mathcal{A}}'''$ is  a pairwise cross $t$-intersecting sequence.
Since $\mathcal{A}_k\backslash \mathcal{A}_k^{-s}\neq \emptyset$ for all $k\in [m]$, we know that $\vec{\mathcal{A}}'''$ is  non-empty.
\begin{claim}\label{cl2}
$ \|\vec{\mathcal{A}}'''\|>\|\vec{\mathcal{A}}\|$.
\end{claim}

\begin{proof}[Proof of Claim \ref{cl2}]
Suppose on the contrary that $\|\vec{\mathcal{A}}'''\|\leq\|\vec{\mathcal{A}}\|$.
Observe that 
\begin{align*}
\left\|\vec{\mathcal{A}}^{-s,\frac{s+t}{2}}(\bar{x}) \right\|= \binom{s}{\frac{s+t}{2}} 
\left\|\vec{\mathcal{A}}_*^{-s,\frac{s+t}{2}}(\bar{x})\right\|
\end{align*}
and 
\[ \left\|\vec{\mathcal{B}}(s,\tfrac{s+t}{2})[x] \right\| =\binom{s}{\frac{s+t}{2}-1} 
\left\|\vec{\mathcal{A}}_*^{-s,\frac{s+t}{2}}[x]\right\|. \]
Moreover, Lemma \ref{le33} implies $\Big\|\vec{\mathcal{A}}_*^{\, -s,\frac{s+t}{2}}(\bar{x}) \Big\|= \Big\|\vec{\mathcal{A}}_*^{\, -s,\frac{s+t}{2}}[x] \Big\|$. Then $\Big\|\vec{\mathcal{A}}_*^{\, -s,\frac{s+t}{2}}(\bar{x}) \Big\|>0$ because $\vec{\mathcal{A}}_*^{\, -s,\frac{s+t}{2}}=\Big(\mathcal{A}_{1,*}^{-s,\frac{s+t}{2}}, \ldots, \mathcal{A}_{m,*}^{-s,\frac{s+t}{2}}\Big) \neq(\emptyset,\ldots, \emptyset)$. 
Recall that $\mathcal{B}(k,s, \frac{s+t}{2})\cap \mathcal{A}_{k}=\emptyset$ for all $k\in[m]$ and $\mathcal{B}(k,s, \frac{s+t}{2})\neq \emptyset$.
Consequently, we get 
$$\|\vec{\mathcal{A}}'''\| 
=\|\vec{\mathcal{A}}\|-\binom{s}{\frac{s+t}{2}} \left\|\vec{\mathcal{A}}_*^{\,-s,\frac{s+t}{2}}(\bar{x}) \right\|+\binom{s}{\frac{s+t}{2}-1} \left\|\vec{\mathcal{A}}_*^{-s,\frac{s+t}{2}}(\bar{x})\right\| 
\leq \|\vec{\mathcal{A}} \|.$$
Therefore, we have $\binom{s}{\frac{s+t}{2}-1}\leq \binom{s}{\frac{s+t}{2}}$, which implies $t\leq 1$, yielding a contradiction.
\end{proof}

It remains to control the extent. Take $x=\ell+1$ in Claim 2.
Let $\vec{\mathcal F}$ be the shifted
componentwise up-set of the resulting sequence $\vec{\mathcal{A}}'''$.
 Then
$\|\vec{\mathcal F}\|>\|\vec{\mathcal A}\|$ and the only possible new
generating coordinate is $\ell+1$. 
Suppose that $\vec{\mathcal F}$ has extent $\ell+1$, and let
$\vec{\mathcal G}$ be its generating sequence. We have
$
  \vec{\mathcal G}[\ell+1]^{(t)}
  =(\emptyset,\ldots,\emptyset).
$
Indeed, no original generating set contains $\ell+1$, while every newly added
set contains $\frac{s+t}{2}-1$ elements of $[s]$ together with both $s+1$ and $\ell+1$,
and hence has size at least $\frac{s+t}{2}+1\ge t+1$. 
Since $\ell+1+t$ is odd,  applying Lemma \ref{le22} (2) removes $\ell+1$ without decreasing the norm. 
Thus, we obtain a non-empty monotone pairwise cross $t$-intersecting sequence
$\vec{\mathcal B}$ with extent at most $\ell$ and
$
  \|\vec{\mathcal B}\|\geq \|\vec{\mathcal F}\|>\|\vec{\mathcal A}\|,
$
as desired.
\end{proof}

\begin{lemma}\label{le34}
Let $\vec{\mathcal{A}} = (\mathcal{A}_1, \ldots, \mathcal{A}_m) \in (2^{[n]})^{m}$ be  a non-empty monotone shifted pairwise cross $t$-intersecting  sequence with $\|\vec{\mathcal{A}}\|$ maximal. Let   $s$ be  the  symmetric extent of $\vec{\mathcal{A}}$. Then $\mathcal{A}_{k}^{-s, i}=\emptyset$ for all $i\in [t-1]$ and $k\in [m]$.
\end{lemma} 
\begin{proof}
If  $\mathcal{A}_{k}^{-s, i}\neq\emptyset$ for some $i\in [t-1]$ and $k\in [m]$, then there exists $A\in\mathcal{A}_{k}^{-s}$ such that $ |A\cap[s]|=i $. By Part (2) of Lemma \ref{le3}, we obtain that  $A_{a,s+1} \notin \mathcal{A}_k$ for any  $a \in A\cap [s]$.
 By the maximality of 
$\|\vec{\mathcal{A}}\|$, there exists $B\in A_{q}$ for some $q\neq k$ such that $|A_{a,s+1}\cap B|<t$. Since $A_{a,s+1}=(A \backslash\{a\}) \cup\{s+1\}$, it follows from  $|A\cap B|\geq  t$ that
$B\cap \{a,s+1\}=\{a\}$ and $|A\cap B|= t$.
This, together with $ |A\cap[s]|=i<t $ and $s+1\notin A\cap B$, implies that there exists $x\in A\cap B\cap [s+2, n]$. By the shiftedness, we have $(B \backslash\{x\}) \cup\{s+1\}\in  A_{q}$.
Therefore, we get $|((B\backslash\{x\}) \cup\{s+1\})\cap A|=|A\cap B|-1= t-1$, which contradicts the pairwise cross $t$-intersecting property of 
$\vec{\mathcal{A}}$.
\end{proof}

\section{Proof of Theorem \ref{ma1}} \label{se3}

Throughout the proof, we denote 
$$ \mathcal{R}(n, \ell)  =\left\{R \subseteq [n]:|R \cap[\ell]| \geq  t\right\} $$ 
and 
$$ \mathcal{S}(n,\ell)  =\left\{S \subseteq [n]: [\ell] \subseteq S\right\}.$$
Clearly, we have $|\mathcal{R}(n, \ell)|=2^{n-\ell}\sum_{k=t}^{\ell}\binom{\ell}{k}$ and $|\mathcal{S}(n,\ell)|=2^{n-\ell}$.

\begin{lemma}\label{le4}
Let $\ell\in [t,n]$  and $f(\ell)= \left(\sum_{k=t}^{\ell}\binom{\ell}{k}  + m - 1\right)2^{n-\ell}$. Then 
$$f(\ell)\leq\max\{f(t), f(n)\}.$$
\end{lemma}

\begin{proof}
For $\ell\in [t,n-1]$, we have 
$$
\frac{f(\ell+1)}{f(\ell)}= \frac{2\sum_{k=t}^{\ell}\binom{\ell}{k}  +\binom{\ell}{t-1}+ m - 1}{2\left(\sum_{k=t}^{\ell}\binom{\ell}{k}  + m - 1\right)}=1+\frac{\binom{\ell}{t-1}-( m - 1)}{2\left(\sum_{k=t}^{\ell}\binom{\ell}{k}  + m - 1\right)}.
$$
It follows that $f(\ell+1)\geq  f(\ell)$ if and only if $\binom{\ell}{t-1}\geq ( m - 1)$. Since $\binom{\ell}{t-1}$ is increasing with respect to $\ell$, we infer that $f(\ell)\leq\max\{f(t), f(n)\}.$
\end{proof}

Now we are ready to prove our main result.

\begin{proof}[Proof of Theorem \ref{ma1}]
Let $ \mathcal{A}_1, \mathcal{A}_2, \ldots, \mathcal{A}_m \subseteq 2^{[n]}$ be non-empty pairwise cross $t$-intersecting families with 
$\sum_{i=1}^{m} |\mathcal{A}_i|$ maximal. 
 By the maximality of 
$\|\vec{\mathcal{A}}\| = \sum_{i=1}^{m} |\mathcal{A}_i|$,   $\vec{\mathcal{A}}$ is monotone.
By Lemma \ref{G23}, we may assume that $\vec{\mathcal{A}} = (\mathcal{A}_1, \ldots, \mathcal{A}_m)$ is shifted. It suffices to prove the theorem under the shifted condition.  In the equality case, a direct inspection of
the inverse shifts of the extremal configurations obtained below shows that
the original sequence is isomorphic to its shifted image.

Let $\vec{\mathcal{G}} = (\mathcal{G}_1, \ldots, \mathcal{G}_m)$ be the generating sequence of $\vec{\mathcal{A}}$.
Furthermore, let $\ell$ be the extent of $\vec{\mathcal{A}}$.
Clearly, we have $\ell\geq  t$. 

For the case $t=1$, we have  $|\mathcal{A}_i|+|\mathcal{A}_{j}|\leq 2^{n}$ for any $i\neq j$ because $\mathcal{A}_i$ and $\mathcal{A}_j$ are non-empty cross intersecting. Then $|\mathcal{A}_i|+|\mathcal{A}_{i+1}|\leq 2^{n}$ for every $i\in [m-1]$, and $|\mathcal{A}_m|+|\mathcal{A}_{1}|\leq 2^{n}$. This yields 
$$2\|\vec{\mathcal{A}}\|=\sum_{i=1}^{m-1}(|\mathcal{A}_i|+|\mathcal{A}_{i+1}|)+|\mathcal{A}_m|+|\mathcal{A}_{1}|\leq m2^{n}. $$
Hence, we get  
$\|\vec{\mathcal{A}}\|\leq m2^{n-1}= m M(n, 1)$, 
as required. 

In what follows, we consider the case $t>1$. 
If $\ell=t$, then $\mathcal{G}_1=\cdots=\mathcal{G}_m=\{[t]\}$.
Thus, we get $\|\vec{\mathcal{A}}\|=m2^{n-t}\leq m M(n, t)$. Moreover, the equality holds if and only if $n=t$ or
$n=t+1$. These give $\mathcal{K}(t,t)$ and $\mathcal{K}'(t+1,t)$, respectively. 
From now on, we assume that $\ell>t>1$.

If $\mathcal{G}_i[\ell]^{(t)}\neq \emptyset$ for some $i\in [m]$, then $[t-1]\cup\{\ell\}\in \mathcal{G}_i\subseteq \mathcal{A}_i$. By the shiftedness, we get 
$
\{[t-1]\cup \{j\}: j\in [t,\ell]\}\subseteq \mathcal{A}_i.
$
Using the cross $t$-intersecting property, we obtain $[\ell] \subseteq B$ for 
any $B\in \mathcal{A}_k$ with $k\neq i$.
By the definition of $\ell$, we infer that $\mathcal{G}_k=\{[\ell]\}$ for all $k\neq i$. It follows from the monotonicity that $\mathcal{A}_k=S(n,\ell)$ for all $k\neq i$. By the maximality of 
$\|\vec{\mathcal{A}}\|$, we get $\mathcal{A}_i=R(n,\ell)$.  
Invoking Lemma \ref{le4}, we deduce that
$$\|\vec{\mathcal{A}}\|=|R(n,\ell)|+(m-1)|S(n,\ell)|\leq  \max \left\{ \sum_{k=t} ^{n}\binom{n}{k}  + m - 1, \, m 2^{n-t} \right\}.$$
Since
$
2^{n-t}\le M(n,t)
$ 
and 
$\ell>t$, the result follows.

Next, we assume that $\vec{\mathcal{G}}[\ell]^{(t)}=(\emptyset,\ldots,\emptyset)$.  
We first prove the theorem in the case where $\ell$ is minimal.
By the maximality of 
$\|\vec{\mathcal{A}}\|$ and  the minimality of  $\ell$, 
it follows from Lemma \ref{le22} that $2\mid (\ell+t)$. Moreover, we may assume that $\vec{\mathcal{G}}[\ell]= \vec{\mathcal{G}}[\ell]^{(\frac{\ell+t}{2})}\neq (\emptyset,\ldots,\emptyset)$.  Then there exists some $i\in [m]$ such that 
$\mathcal{G}_i[\ell]=\mathcal{G}_i[\ell]^{(\frac{\ell+t}{2})}\neq \emptyset$. 
Let  $s$ be the symmetric extent of $\vec{\mathcal{A}}$.
Then Lemma \ref{le34} implies $\mathcal{A}_{k}^{-s, i}=\emptyset$ for all $i\in [t-1]$ and $k\in [m]$.

There are two cases in the remaining proof: $s\geq  \ell$; or $s< \ell$. 

{\bf Case 1.}~  If $s\geq  \ell$, then $\mathcal{G}_i=\binom{[\ell]}{\frac{\ell+t}{2}}$. 
This implies that
$\mathcal{A}_i  =\left\{A \subseteq [n]:|A \cap[\ell]| \geq  \frac{\ell+t}{2}\right\}.$
Since $\vec{\mathcal{A}}$ is pairwise cross $t$-intersecting and $\|\vec{\mathcal{A}}\|$ is maximal, we infer that $\mathcal{A}_k  =\{A \subseteq [n]:|A \cap[\ell]| \geq$ $\frac{\ell+t}{2} \}$ for all $k\neq i$.
We define 
$g(\ell)= |\left\{A \subseteq [n]:|A \cap[\ell]| \geq  \frac{\ell+t}{2}\right\}|$ for $\ell\in [t+1,n]$.
By a simple calculation, we obtain 
$$
g(\ell)\leq \begin{cases}g(n)& \text { if } 2\mid (n+t); \\ g(n-1) &  \text { if } 2\nmid (n+t).\end{cases}
$$ 
It is immediate that 
$\|\vec{\mathcal{A}}\|= m g(\ell)\leq  m M(n, t)$, 
where the equality holds if and only if $\mathcal{A}_1=\cdots= \mathcal{A}_m= \mathcal{K}(n,t)$ for $2\mid (n+t)$  and $\mathcal{A}_1=\cdots= \mathcal{A}_m= \mathcal{K}'(n,t)$ for $2\nmid (n+t)$.

{\bf Case 2.}~  If $s< \ell$, then applying Lemma \ref{le5} leads to $\ell=n$. 
For every $i\in[m]$, let $\mathcal{A}_i'=\{A\cup \{n+1\}: A\in \mathcal{A}_i \}$ and 
$\mathcal{A}_i''=\mathcal{A}_i\cup \mathcal{A}_i'$.
Then $\vec{\mathcal{A}}'' = (\mathcal{A}_1'', \ldots, \mathcal{A}_m'')\in (2^{[n+1]})^{m}$ is a non-empty monotone shifted pairwise cross $t$-intersecting sequence with its extent $n$ and symmetric extent $s$.  Furthermore, we have $\mathcal{A}_{k}''^{-s, i}=\emptyset$ for all $i\in [t-1]$ and $k\in [m]$.
Since $s<n<n+1$, $t>1$ and $2\mid (n+t)$,  we can  apply Lemma \ref{le5} to obtain a non-empty monotone pairwise cross $t$-intersecting sequence $\vec{\mathcal{B}} \in (2^{[n+1]})^{m}$ with its extent at most $n$ satisfying $\|\vec{\mathcal{A}}''\|<\|\vec{\mathcal{B}}\|$.
Since the extent of $\vec{\mathcal{B}}$ is at most $n$, we infer that $\|\vec{\mathcal{B}}[n+1]\|\leq\|\vec{\mathcal{B}}(\overline{n+1})\|$. 
Therefore, we have 
$$\|\vec{\mathcal{B}}(\overline{n+1})\|\geq  \frac{1}{2}\|\vec{\mathcal{B}}\|> \frac{1}{2}\|\vec{\mathcal{A}}''\|=\|\vec{\mathcal{A}}\|.$$
Moreover, we see that $\vec{\mathcal{B}}(\overline{n+1})\in (2^{[n]})^{m}$
is a non-empty  pairwise cross $t$-intersecting sequence,  which leads to a contradiction with the maximality of $\vec{\mathcal{A}}$. 

Having completed the proof for the case where $\ell$ is minimal and $\vec{\mathcal{G}}[\ell]= \vec{\mathcal{G}}[\ell]^{(\frac{\ell+t}{2})}\neq (\emptyset,\ldots,\emptyset)$, we now consider the case where $\ell$ is minimal and $\vec{\mathcal{G}}[\ell]\neq \vec{\mathcal{G}}[\ell]^{(\frac{\ell+t}{2})}$.
Recall that $\vec{\mathcal{G}}[\ell]^{(t)}=(\emptyset,\ldots,\emptyset)$ and $\ell>t$. 
By Lemma \ref{le22} and the minimality of $\ell$, we have that $2\mid (\ell+t)$ and 
there exists a non-empty monotone shifted pairwise cross $t$-intersecting sequence $\vec{\mathcal{B}}$  with  extent $\ell$ and 
generating sequence $\vec{\mathcal{G}}'$ 
such that $\|\vec{\mathcal{B}}\|\geq \|\vec{\mathcal{A}}\|$ and $\vec{\mathcal{G}}'[\ell]= \vec{\mathcal{G}}[\ell]^{(\frac{\ell+t}{2})}$.
Applying the above conclusion to  $\vec{\mathcal{G}}'$ yields 
$\vec{\mathcal{G}}[\ell]^{(\frac{\ell+t}{2})}=\Big(\binom{[\ell]}{\frac{\ell+t}{2}}[\ell], \ldots,\binom{[\ell]}{\frac{\ell+t}{2}}[\ell] \Big)$.
Since $\mathcal{G}_i$ is an antichain for each $i\in [m]$, we deduce that $\vec{\mathcal{G}}[\ell]= \vec{\mathcal{G}}[\ell]^{(\frac{\ell+t}{2})}$, a contradiction.
We have shown that  when $\ell$ is minimal, $\ell=n$ if $2\mid (n+t)$ and $\ell=n-1$ if $2\nmid (n+t)$.
If $\ell$ is not minimal, then $\ell=n$ and $2\nmid (n+t)$.
Thus, we have $ \vec{\mathcal{G}}= \vec{\mathcal{A}}$.
Since $2\nmid (n+t)$ and $\|\vec{\mathcal{A}}\|$ is maximal, 
 Lemma \ref{le22} implies that  for all pairs $u, v\in [t,n]$ with  $u+v=n+t$, either $\vec{\mathcal{G}}[n]^{(u)}\neq (\emptyset,\ldots, \emptyset)$ and $\vec{\mathcal{G}}[n]^{(v)}\neq (\emptyset,\ldots, \emptyset)$, or $\vec{\mathcal{G}}[n]^{(u)}=\vec{\mathcal{G}}[n]^{(v)}=(\emptyset,\ldots, \emptyset)$. Furthermore, there exists a non-empty pairwise cross $t$-intersecting sequence $\vec{\mathcal{B}}$  with extent at most $n-1$  such that $\|\vec{\mathcal{B}}\|\geq \|\vec{\mathcal{A}}\|$.
Moreover,  the generating sequence of  $\vec{\mathcal{B}}$ is constructed by repeatedly applying the following operation to $\vec{\mathcal{G}}$: 
for all pairs 
 $u, v\in [t,n]$ with  $u+v=n+t$ such that both $\vec{\mathcal{G}}[n]^{(u)}$ and $\vec{\mathcal{G}}[n]^{(v)}$ are  not equal to  $(\emptyset,\ldots, \emptyset)$, remove $\vec{\mathcal{G}}[n]^{(u)}$ and $ \vec{\mathcal{G}}[n]^{(v)}$, and add either $\vec{\mathcal{G}}[n]^{(v)}(n)$ or $\vec{\mathcal{G}}[n]^{(u)}(n)$.
All options must lead eventually to the same  sequence $\vec{\mathcal{B}}=(\mathcal{K}'(n,t), \ldots, \mathcal{K}'(n,t))$, and this only happen if $\vec{\mathcal{G}}[n]^{(u)}=\vec{\mathcal{G}}[n]^{(v)}=(\emptyset,\ldots, \emptyset)$ for all pairs 
 $u, v\in [t,n]$ with  $u+v=n+t$, contradicting with $\vec{\mathcal{G}}[n]\neq (\emptyset,\ldots, \emptyset)$.
\end{proof}

\section*{Acknowledgement}
The authors would like to thank Yang Huang and Yuejian Peng for sharing the references \cite{Huang2025} prior to its publication.  
The authors also thank the referee for the valuable comments  which improved the presentation of the manuscript. 
Y. Li as the corresponding author was supported by the program of Tsinghua University,
 T. Wu as the corresponding author was supported by NSFC (No. 12261071) and NSF of Qinghai Province (No. 2025-ZJ-902T),  
 and L. Feng was supported by the NSFC (Nos. 12271527 and 12471022).

\end{document}